\documentclass[12pt, reqno]{amsart}
\usepackage{amsmath, amsthm, amscd, amsfonts, amssymb, graphicx, color}
\usepackage[bookmarksnumbered, colorlinks, plainpages]{hyperref}
\usepackage{cite}
\theoremstyle{definition}

\newtheorem{thm}{Theorem}[section]

\newtheorem{cor}{Corollary}[section]

\newtheorem{lem}{Lemma}[section]

\newtheorem{defi}{Definition}[section]

\newcommand{\be}{\begin{equation}}
	\newcommand{\ee}{\end{equation}}
\newcommand{\beas}{\begin{eqnarray*}}
	\newcommand{\eeas}{\end{eqnarray*}}
\newcommand{\bea}{\begin{eqnarray}}
	\newcommand{\eea}{\end{eqnarray}}
\numberwithin{equation}{section}
\begin{document}
\setcounter{page}{1}
\title[Coefficient bounds for starlike functions with Gregory coefficients]{Coefficient bounds for starlike functions associated with Gregory coefficients}

\author{Molla Basir Ahamed}
\address{Molla Basir Ahamed, Department of Mathematics, Jadavpur University, Kolkata-700032, West Bengal, India.}
\email{mbahamed.math@jadavpuruniversity.in}

\author{Sanju Mandal}
\address{Sanju Mandal, Department of Mathematics, Jadavpur University, Kolkata-700032, West Bengal, India.}
\email{sanjum.math.rs@jadavpuruniversity.in, sanju.math.rs@gmail.com}
	
\begin{abstract} 
 It is of interest to know the sharp bounds of the Hankel determinant, Zalcman functionals, Fekete–Szeg$ \ddot{o} $ inequality as a part of coefficient problems for different classes of functions. Let $\mathcal{H}$ be the class of functions $ f $ which are holomorphic in the open unit disk $\mathbb{D}=\{z\in\mathbb{C}: |z|<1\}$ of the form 
 \begin{align*}
 	f(z)=z+\sum_{n=2}^{\infty}a_nz^n\; \mbox{for}\; z\in\mathbb{D}
 \end{align*}
 and suppose that
 \begin{align*}
 	F_{f}(z):=\log\dfrac{f(z)}{z}=2\sum_{n=1}^{\infty}\gamma_{n}(f)z^n, \;\; z\in\mathbb{D},\;\;\log 1:=0,
 \end{align*}
 where $ \gamma_{n}(f) $ is the logarithmic coefficients. The second Hankel determinant of logarithmic coefficients $H_{2,1}(F_{f}/2)$ is defined as: $H_{2,1}(F_{f}/2) :=\gamma_{1}\gamma_{3} -\gamma^2_{2}$, where $\gamma_1, \gamma_2,$ and $\gamma_3$ are the first, second and third logarithmic coefficients of functions belonging to the class $\mathcal{S}$ of normalized univalent functions. In this article, we first establish sharp inequalities $|H_{2,1}(F_{f}/2)|\leq 1/64$ with logarithmic coefficients for the classes of starlike functions associated with Gregory coefficients. In addition, we establish the sharpness of Fekete–Szeg$ \ddot{o} $ inequality, Zalcman functional and generalized Zalcman functional for the class starlike functions associated with Gregory coefficients.
\end{abstract} \maketitle
	
\section{\bf Introduction}
Let $\mathcal{H}$ be the class of holomorphic functions $f$ in the open unit disk $\mathbb{D}=\{z\in\mathbb{C}: |z|<1\}$. Then $\mathcal{H}$ is a locally convex topological vector space endowed with the topology of uniform convergence over compact subsets of $\mathbb{D}$. Let $\mathcal{A}$ denote the class of functions $f\in\mathcal{H}$ such that $f(0)=0$ and $f^{\prime}(0)=1$ \textit{i.e}, the function $f$ is of the form
\begin{align}\label{eq-1.1}
	f(z)=z+ \sum_{n=2}^{\infty}a_nz^n,\; \mbox{for}\; z\in\mathbb{D}.
\end{align} 
Let $\mathcal{S}$ denote the subclass of all functions in $\mathcal{A}$ which are univalent. For a comprehensive understanding of univalent functions and their significance in coefficient problems, readers are referred to the following books \cite{Duren-1983-NY,Goodman-1983}.

\begin{defi}\label{def-1.1}
	Let $f$ and $g$ be two analytic functions in $\mathbb{D}$. Then $f$ is said to be subordinate to $g$, written as $f\prec g$ or $f(z)\prec g(z)$, if there exists a function $\omega$, analytic in $\mathbb{D}$ with $w(0)=0$, $|w(z)|<1$ such that $f(z)=g(w(z))$ for $z\in\mathbb{D}$. Moreover, if $g$ is univalent in $\mathbb{D}$ and $f(0)=g(0)$, then $f(\mathbb{D})\subseteq g(\mathbb{D})$.
\end{defi}

In \cite{Ma-Minda-1994}, Ma and Minda gave a unified presentation of various subclasses of starlike and convex functions by replacing the subordinate function $(1+z)/(1-z)$ by a more general analytic function $\varphi$ with positive real part and normalized by the conditions $\varphi(0) =1$, $\varphi^{\prime}(0)> 0$ and $\varphi$ maps $\mathbb{D}$ onto univalently a starlike region with respect to $1$ and symmetric with respect to the real axis. They have introduced a general class that envelopes several well-known classes as special cases
\begin{align*}
	\mathcal{S}^*[\varphi]=\{f\in\mathcal{A}: zf^{\prime}(z)/f(z)\prec\varphi(z)\}.
\end{align*} In the literature, functions belonging to the class $\mathcal{S}^*[\varphi]$ are known as the Ma-Minda starlike functions. For $-1\leq B<A\leq 1$, the class $\mathcal{S}^*[(1+Az)/(1+Bz)]:= \mathcal{S}^*[A,B]$
is called the class of Janowski starlike functions, introduced by Janowski in \cite{Janowski-APM-1973}. The class $\mathcal{S}^*[\beta]$ of starlike functions of order $\beta $, where $ 0\leq\beta<0 $, is defined by taking $\varphi(z)= (1+(1-2\beta)z)/(1-z)$. Note that $\mathcal{S}^*=\mathcal{S}^*[0]$ is the classical class of starlike functions. By taking 
\begin{align*}
	\varphi(z)= 1+ \frac{2}{\pi^2}\left(\log\left((1+\sqrt{z})/(1-\sqrt{z})\right)\right)^2,
\end{align*} we obtain the class $\mathcal{S}^*[\varphi]=\mathcal{S}_p$ of parabolic starlike functions, introduced in \cite{Rønning-PAMS-1993}. \vspace{1.2mm}

Recently, the coefficient problem and many other geometric properties for class $\mathcal{S}^*[\varphi]$ are studied extensively in \cite{Deniz-BMMSS-2021,Riza-Raza-Thomas-FM-2022, Goodman-APM-1991,Kazimoglu-Deniz-Srivastava-COAT-2024}. For example, in \cite{Deniz-BMMSS-2021}, Deniz has studied sharp coefficient problem for the function
\begin{align*}
	\varphi(z)=e^{z+\frac{\lambda}{2}z^2} \;\;(z\in\mathbb{C}, \lambda\geq 1), 
\end{align*}
which is a generated function of generalized telephone numbers. In \cite{Mendiratta-Nagpal-Ravichandran-BMMSS-2015}, Mendiratta \textit{et al.}  obtained the structural formula, inclusion relations, coefficient estimates, growth and distortion results, subordination theorems and various radii constants for the exponential function $ 	\varphi(z)= e^z. $\vspace{1.2mm}

It is worth pointing out that, the case when $\varphi$ defined by
\begin{align*}
	\varphi(z)=\frac{2}{1+e^{-z}} 
\end{align*}
is a modified sigmoid function which maps the unit disk $\mathbb{D}$ onto the domain
\begin{align*}
	\Delta_{SG}=\left\{z\in\mathbb{C}\;:\;\vline\log\left(\frac{z}{2-z}\right)\vline< 1\right\}.
\end{align*}
Sharp coefficient problem is studied for this function by Riza \textit{et al.} in \cite{Riza-Raza-Thomas-FM-2022}, Goel and Kumar in \cite{Goel-Kumar-BMMSS-2020}. In \cite{Kazimoglu-Deniz-Srivastava-COAT-2024}, Kazımoglu \textit{et al.} have considered the function $\varphi$ for which $\varphi(\mathbb{D})$ is starlike with respect to $1$ and whose coefficients is the Gregory coefficients. Gregory
coefficients are decreasing rational numbers of the form $1/2, 1/12, 1/24, 19/720, \ldots$, which are similar in their role to the Bernoulli numbers and appear in a large number of problems, especially in those related to the numerical analysis and to the number theory. They first appeared in the works of the Scottish mathematician James Gregory in 1671 and have since been rediscovered numerous times. Among the renowned mathematicians who rediscovered them are Laplace, Mascheroni, Fontana, Bessel, Clausen, Hermite, Pearson, and Fisher. Furthermore, Gregory's coefficients can rightfully be considered among the most frequently rediscovered entities in mathematics, with the most recent rediscovery dating back to our century. This frequent rediscovery has led to various names for them in literature, such as reciprocal logarithmic numbers, Bernoulli numbers of the second kind, Cauchy numbers, and more. Consequently, their authorship is often attributed to various mathematicians.\vspace{1.2mm}

In \cite{Berezin-Zhidkov-1965,Phillips-AMM-1972}, the authors have considered the generating function of the Gregory coefficients $G_n$  as follows:
\begin{align*}
	\frac{z}{\ln(1+z)}=\sum_{n=0}^{\infty} G_n z^n\; \mbox{for}\; z\in\mathbb{D}.
\end{align*}
\begin{figure}[!htb]
	\begin{center}
		\includegraphics[width=0.85\linewidth]{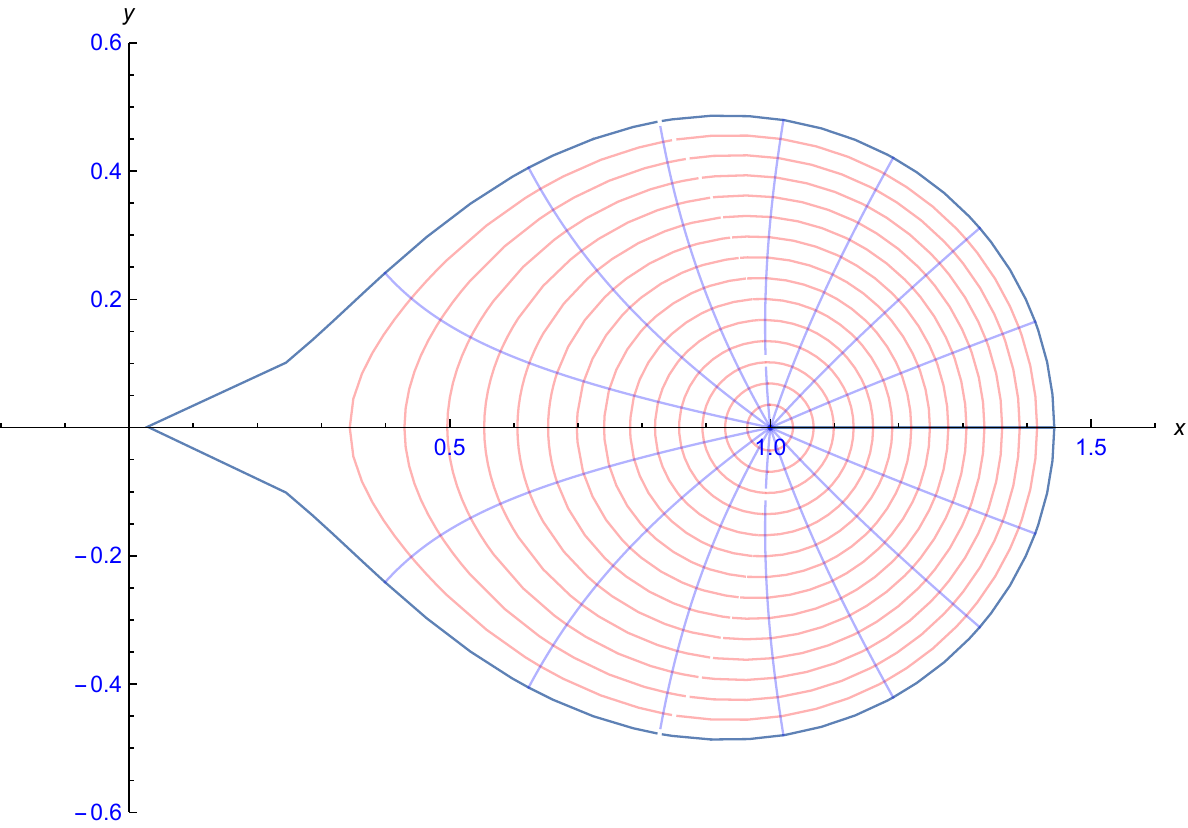}
	\end{center}
	\caption{The image $\Psi(\mathbb{D})$ which is starlike with respect to $1$.}
\end{figure}
For first few positive integers, the reader may check that $G_0=1$, $G_1=\frac{1}{2}$, $G_2=-\frac{1}{12}$, $G_3=\frac{1}{24}$, $G_4=-\frac{19}{720}$, $G_5= \frac{3}{160}$ and $G_6=-\frac{863}{60480}$. In \cite{Kazimoglu-Deniz-Srivastava-COAT-2024}, Kazımoglu \textit{et al.} have considered the function $\Psi(z):=\frac{z}{\ln(1+z)}$ with its domain of definition as the open unit disk $\mathbb{D}$. With this function, we define the class
\begin{align*}
	S^*_G:=\{f: f\in\mathcal{A} \;\;\mbox{and}\;\; zf^{\prime}(z)/f(z) \prec\Psi(z)\}.
\end{align*}
Finding the upper bound for coefficients has been one of the central topics of research in geometric function theory as it gives several properties of functions. The main hurdle lies in finding a suitable function from the class that convincingly exhibits the sharpness of the bound.  Keeping in mind such hurdle, we investigate the sharp bound of several problems in geometric function theory such as finding sharp bound of Hankel determinant of logarithmic coefficients, sharpness of the Fekete-Szeg$ \ddot{o} $ inequality, Zalcman functionals. In the subsequent sections, we discuss our findings and background study of individuals.
\section{\bf Sharp bound of a Hankel determinant of logarithmic coefficients for functions in the class $\mathcal{S}^*_G$}
For $f\in\mathcal{S}$, we define
\begin{align}\label{eq-2.1}
	F_{f}(z):=\log\dfrac{f(z)}{z}=2\sum_{n=1}^{\infty}\gamma_{n}(f)z^n, \;\; z\in\mathbb{D},\;\;\log 1:=0,
\end{align}
a logarithmic function associated with $f\in\mathcal{S}$. The numbers $\gamma_{n}:=\gamma_{n}(f)$ ($ n\in\mathbb{N} $) are called the logarithmic coefficients of $f$. Although the logarithmic coefficients $\gamma_{n}$ play a critical role in the theory of univalent functions, it appears that only a limited number of sharp bounds have been established for them. As is well known, the logarithmic coefficients play a crucial role in Milin’s conjecture (\cite{Milin-1977-ET}, see also \cite[p.155]{Duren-1983-NY}).
Milin conjectured that for $f\in\mathcal{S}$ and $n\geq2$,
\begin{align*}
	\sum_{m=1}^{n}\sum_{k=1}^{m}\left(k|\gamma_{k}|^2 -\frac{1}{k}\right)\leq 0,
\end{align*}
where the equality holds if, and only if, $f$ is a rotation of the Koebe function. De Branges \cite{Branges-AM-1985} proved Milin conjecture which confirmed the famous Bieberbach conjecture. On the other hand, one of reasons for more attention has been given to the logarithmic coefficients is that the sharp bound for the class $\mathcal{S}$ is known only for $\gamma_{1}$ and $\gamma_{2}$, namely
\begin{align*}
	|\gamma_{1}|\leq 1, \;\; |\gamma_{2}|\leq \dfrac{1}{2}+ \dfrac{1}{e} =0.635\ldots
\end{align*}
It is still an open problem to find the sharp bounds of $\gamma_{n}$, $n\geq 3$, for the class $\mathcal{S}$. Estimating the modulus of logarithmic coefficients for $f\in\mathcal{S}$ and various sub-classes has been considered recently by several authors. For more information on the topic, we recommend consulting the articles \cite{Ali-Allu-PAMS-2018, Roth-PAMS-2007, Ali-Allu-Thomas-CRMCS-2018, Cho-Kowalczyk-kwon-Lecko-Sim-RACSAM-2020,Girela-AASF-2000,Thomas-PAMS-2016} and references therein.\vspace{2mm}

The evaluation of Hankel determinants has been a major concern in geometric function theory, where these determinants are formed by employing coefficients of analytic functions $f$ that are characterized by \eqref{eq-1.1} and defined within the region $\mathbb{D}$. Hankel matrices (and determinants) have emerged as fundamental elements in different areas of mathematics, finding a wide array of applications (see \cite{Ye-Lim-FCM-2016}). The primary objective of this study is to determine the sharp bound for the Hankel determinants, which involves the use of logarithmic coefficients. To begin, we present the definitions of Hankel determinants in situations where $f\in \mathcal{A}$.\vspace{1.2mm}

The Hankel determinant $H_{q,n}(f)$ of Taylor's coefficients of functions $f\in\mathcal{A}$ represented by \eqref{eq-1.1} is defined for $q,n\in\mathbb{N}$ as follows:
\begin{align*}
	H_{q,n}(f):=\begin{vmatrix}
		a_{n} & a_{n+1} &\cdots& a_{n+q-1}\\ a_{n+1} & a_{n+2} &\cdots& a_{n+q} \\ \vdots & \vdots & \vdots & \vdots \\ a_{n+q-1} & a_{n+q} &\cdots& a_{n+2(q-1)}
	\end{vmatrix}.
\end{align*}
The extensive exploration of sharp bounds of the Hankel determinants for starlike, convex, and other function classes has been undertaken in various studies (see \cite{Ponnusamy-Sugawa-BDSM-2021, Kowalczyk-Lecko-RACSAM-2023, Raza-Riza-Thomas-BAMS-2023, Sim-Lecko-Thomas-AMPA-2021, Kowalczyk-Lecko-BAMS-2022,Mandal-Ahamed-LMJ-2024}), and their precise bounds have been successfully established. \vspace{1.2mm}

Differentiating \eqref{eq-2.1} and using \eqref{eq-1.1}, a simple computation shows that 
\begin{align}\label{Eq-2.2}
	\begin{cases}
		\gamma_{1}=\dfrac{1}{2}a_{2},\vspace{1.5mm}\\ \gamma_{2}=\dfrac{1}{2} \left(a_{3} -\dfrac{1}{2}a^2_{2}\right), \vspace{1.5mm}\\ \gamma_{3} =\dfrac{1}{2}\left(a_{4}- a_{2}a_{3} +\dfrac{1}{3}a^3_{2}\right) \vspace{1.5mm}.
	\end{cases}
\end{align}
In \cite{Kowalczyk-Lecko-BAMS-2022}, Kowalczyk and Lecko  proposed a Hankel determinant $H_{q,n}(F_f/2)$ whose elements are the logarithmic coefficients of $f\in\mathcal{S}$, realizing the extensive use of these coefficients. Thus, it follows that
\begin{align}\label{eq-2.2}
	H_{2,1}(F_f/2):=\gamma_1\gamma_3-\gamma_2^2=\frac{1}{48}\left(a_2^4-12a_3^2+12a_2a_4\right).
\end{align}
Furthermore, $H_{2,1}(F_{f}/2)$ is invariant under rotation, since for $f_{\theta}(z):=e^{-i\theta}f(e^{i\theta}z)$, $\theta\in\mathbb{R}$ when $f\in\mathcal{S}$, we have
\begin{align*}
	H_{2,1}(F_{f_{\theta}}/2)=\frac{e^{4i\theta}}{48}\left(a^4_2 - 12 a^2_3 + 12 a_2 a_4\right)=e^{4i\theta}H_{2,1}(F_{f}/2).
\end{align*}
The class $\mathcal{P}$ of all analytic functions $p$ in $\mathbb{D}$ satisfying $p(0)=1$ and $\mbox{Re}\;p(z)>0$ for $z\in\mathbb{D}$. Thus, every $p\in\mathcal{P}$ can be represented as
\begin{align}\label{eq-2.3}
	p(z)=1+\sum_{n=1}^{\infty}c_n z^n,\; z\in\mathbb{D}.
\end{align}
Elements of the class $\mathcal{P}$ are called  Carath$\acute{e}$odory functions. It is known that $|c_n|\leq 2$, $n\geq 1$ for a function $p\in\mathcal{P}$ (see \cite{Duren-1983-NY}). The Carath$\acute{e}$odory class $\mathcal{P}$ and it's coefficients bound plays a significant roles in establishing the bound of Hankel determinants.\vspace{2mm}

In this section, we provide important lemmas \emph{i.e.,} Lemma \ref{lem-2.1} and Lemma \ref{lem-2.2}, that will be utilized to establish the main result. Parametric representations of the coefficients are often useful in finding the bound of Hankel determinants, and in this regard, Libera and Zlotkiewicz (see \cite{Libera-Zlotkiewicz-PAMS-1982, Libera-Zlotkiewicz-PAMS-1983}) obtained the parameterizations of possible values of $c_2$ and $c_3$.
\begin{lem}(\cite{Libera-Zlotkiewicz-PAMS-1982,Libera-Zlotkiewicz-PAMS-1983})\label{lem-2.1}
	If $p\in\mathcal{P}$ is of the form \eqref{eq-2.3} with $c_1\geq 0$, then 
	\begin{align}\label{eq-2.4}
		&c_1=2\tau_1,\\\label{eq-2.5} &c_2=2\tau^2_1 +2(1-\tau^2_1)\tau_2
	\end{align}
	and
	\begin{align}\label{eq-2.6}
		c_3= 2\tau^3_1  + 4(1 -\tau^2_1)\tau_1\tau_2 - 2(1 - \tau^2_1)\tau_1\tau^2_2 + 2(1 - \tau^2_1)(1 - |\tau_2|^2)\tau_3
	\end{align}
	for some $\tau_1\in[0,1]$ and $\tau_2,\tau_3\in\overline{\mathbb{D}}:= \{z\in\mathbb{C}:|z|\leq 1\}$.\vspace{1.2mm}
	
	For $\tau_1\in\mathbb{T}:=\{z\in\mathbb{C}:|z|=1\}$, there is a unique function $p\in\mathcal{P}$ with $c_1$ as in \eqref{eq-2.4}, namely
	\begin{align*}
		p(z)=\frac{1+\tau_1 z}{1-\tau_1 z}, \;\;z\in\mathbb{D}.
	\end{align*}
	
	For $\tau_1\in\mathbb{D}$ and $\tau_2\in\mathbb{T}$, there is a unique function $p\in\mathcal{P}$ with $c_1$ and $c_2$ as in \eqref{eq-2.4} and \eqref{eq-2.5}, namely
	\begin{align*}
		p(z)=\frac{1+(\overline{\tau_1}\tau_2 +\tau_1)z+\tau_2 z^2}{1 +(\overline{\tau_1}\tau_2 -\tau_1)z-\tau_2 z^2}, \;\;z\in\mathbb{D}.
	\end{align*}
	
	For $\tau_1,\tau_2\in\mathbb{D}$ and $\tau_3\in\mathbb{T}$, there is a unique function $p\in\mathcal{P}$ with $c_1,c_2$ and $c_3$ as in \eqref{eq-2.4}--\eqref{eq-2.6}, namely
	\begin{align*}
		p(z)=\frac{1+(\overline{\tau_2}\tau_3+\overline{\tau_1}\tau_2 +\tau_1)z +(\overline{\tau_1}\tau_3+ \tau_1\overline{\tau_2}\tau_3 +\tau_2)z^2 +\tau_3 z^3}{1 +(\overline{\tau_2}\tau_3+ \overline{\tau_1}\tau_2 -\tau_1)z +(\overline{\tau_1}\tau_3- \tau_1\overline{\tau_2}\tau_3 -\tau_2)z^2 -\tau_3 z^3}, \;\;z\in\mathbb{D}.
	\end{align*}
\end{lem}
\begin{lem}(\cite{Cho-Kim-Sugawa-JMSJ-2007})\label{lem-2.2}
	Let $A, B, C$ be real numbers and
	\begin{align*}
		Y(A,B,C):=\max\{|A+ Bz +Cz^2| +1-|z|^2: z\in\overline{\mathbb{D}}\}.
	\end{align*}
	\noindent{(i)} If $AC\geq 0$, then
	\begin{align*}
		Y(A,B,C)=\begin{cases}
			|A|+|B|+|C|, \;\;\;\;\;\;\;\;\;\;\;\;\;|B|\geq 2(1-|C|), \vspace{2mm}\\ 1+|A|+\dfrac{B^2}{4(1-|C|)}, \;\;\;\;\;|B|< 2(1-|C|).
		\end{cases}
	\end{align*}
	\noindent{(ii)} If $AC<0$, then
	\begin{align*}
		Y(A,B,C)=\begin{cases}
			1-|A|+\dfrac{B^2}{4(1-|C|)}, \;\;\;\;-4AC(C^{-2}-1)\leq B^2\land|B|< 2(1-|C|),\vspace{2mm} \\ 1+|A|+\dfrac{B^2}{4(1+|C|)}, \;\;\;\; B^2<\min\{4(1+|C|)^2,-4AC(C^{-2}-1)\}, \vspace{2mm} \\ R(A,B,C), \;\;\;\;\;\;\;\;\;\;\;\;\;\;\;\;\;\;\; otherwise,
		\end{cases}
	\end{align*}
	where
	\begin{align*}
		R(A,B,C):= \begin{cases}
			|A|+|B|-|C|, \;\;\;\;\;\;\;\;\;\;\;\;\; |C|(|B|+4|A|)\leq |AB|, \vspace{2mm}\\ -|A|+|B|+|C|, \;\;\;\;\;\;\;\;\;\;\; |AB|\leq |C|(|B|-4|A|), \vspace{2mm}\\ (|C| +|A|)\sqrt{1-\dfrac{B^2}{4AC}}, \;\;\; otherwise.
		\end{cases}
	\end{align*}
\end{lem}
The significance of logarithmic coefficients in geometric function theory has led to a growing interest in finding sharp bound of Hankel determinants with these coefficients. We obtain the following sharp bound of $H_{2,1}(F_f/2)$ for the class $\mathcal{S}^*_G$.
\begin{thm}\label{Th-2.1}
	Let $f(z)=z+a_2z^2+a_3z^3+\cdots\in\mathcal{S}^*_G$ and $ \gamma_{1}, \gamma_{2} $, $ \gamma_{3} $ are given by \eqref{Eq-2.2}. Then we have 
	\begin{align*}
		|H_{2,1}(F_f/2)|:=|\gamma_1\gamma_3-\gamma_2^2|\leq \frac{1}{64}.
	\end{align*}
	The inequality is sharp.
\end{thm}
\begin{proof}
	Let $f\in \mathcal{S}^*_{G}$. Then there exists an analytic function $\omega$ with $\omega(0)=0$ and $|\omega(z)|<1$ in $\mathbb{D}$ such that 
	\begin{align}\label{eq-2.7}
		\frac{zf^{\prime}(z)}{f(z)}=\Psi(\omega(z))=\frac{\omega(z)}{\ln(1+\omega(z))},\;\;\;\; z\in\mathbb{D}.
	\end{align}
	Let $p\in\mathcal{P}$. Then, using the definition of subordination, we have
	\begin{align*}
		p(z)=\frac{1+\omega(z)}{1-\omega(z)}=1+c_1z+c_2z^2+\cdots,\; z\in\mathbb{D}.
	\end{align*} 
	Hence, it is evident that
	\begin{align}\label{eq-2.8}
		\omega(z)&=\nonumber\frac{p(z)-1}{p(z)+1}\\&\nonumber=\frac{c_1}{2}z+\frac{1}{2}\left(c_2-\frac{c_1^2}{2}\right)z^2+\frac{1}{2}\left(c_3-c_1c_2+\frac{c_1^3}{4}\right)z^3\\&\quad+\frac{1}{2}\left(c_4-c_1c_3+\frac{3c_1^2c_2}{4}-\frac{c_2^2}{2}-\frac{c_1^4}{8}\right)z^4+\cdots
	\end{align}
	in $\mathbb{D}$. Then $p$ is analytic in $\mathbb{D}$ with $p(0)=1$ and has positive real part in $\mathbb{D}$. In view of \eqref{eq-2.8} together with $\Psi(\omega(z))$, a tedious computation shows that 
	\begin{align}\label{eq-2.9}
		\Psi(\omega(z))&=1+\frac{c_1}{4}z+\frac{1}{48}\left(-7c_1^2+12c_2\right)z^2+\frac{1}{192}\left(17c_1^3-56c_1c_2+48c_3\right)z^3\nonumber\\&\quad+\frac{1}{11520}\left(-649c_1^4+3060c_2^2c_2-3360c_1c_3-1680c_2^2+2880c_4\right)z^4+\cdots.
	\end{align}
	Further, 
	\begin{align}\label{eq-2.10}
		\frac{zf^{\prime}(z)}{f(z)}&\nonumber= 1+a_2 z +(-a^2_{2} +2a_3)z^2 +(a^3_{2} -3a_2 a_3 +3a_4)z^3 \\& \quad+(-a^4_{2} +4a^2_{2}a_3 -2a^2_{3} -4a_2 a_4 +4a_5)z^4 +\cdots.
	\end{align}
	Hence, by \eqref{eq-2.7}, \eqref{eq-2.9}, and \eqref{eq-2.10}, we obtain
	\begin{align}\label{eq-2.11}
		\begin{cases}
			a_2=\dfrac{c_1}{4},\vspace{2mm}\\
			a_3=\dfrac{1}{24}\left(3c_2-c_1^2\right),\vspace{2mm}\\
			a_4=\dfrac{1}{288}\left(4c_1^3-19c_1c_2+24c_3\right),\vspace{2mm}\\
			a_5=-\dfrac{1}{11520}\left(71c_1^4+330c_2^2+600c_1c_3-425c_1^2c_2-720c_4\right).
		\end{cases}
	\end{align}
	Using \eqref{eq-2.2} and \eqref{eq-2.11}, a standard computation leads to
	\begin{align}\label{eq-2.12}
		H_{2,1}(F_f/2)&\nonumber=\frac{1}{48}\left(a_2^4-12a_3^2+12a_2a_4\right)\\&=\frac{1}{36864}\left(19c_1^4 -56c_1^2 c_2 -144c_2^2 +192c_1 c_3\right).
	\end{align}
	By Lemma \ref{lem-2.1} and \eqref{eq-2.12}, we obtain
	\begin{align}\label{eq-2.13}
		H_{2,1}(F_f/2)\nonumber&=\frac{1}{2304} \bigg(3\tau_1^4-4\tau_1^2\tau_2 (1-\tau_1^2)-12\tau_2^2(1-\tau_1^2)(3+\tau_1^2)\\&\quad+48\tau_1\tau_3(1-\tau_1^2)(1-|\tau_2|^2)\bigg).
	\end{align}
	
	We now explore three possible cases involving $\tau_1$. \vspace{1.2mm}
	
	\noindent{\bf Case-I.} Let $\tau_1=1$. Then, from \eqref{eq-2.13} we see that 
	\begin{align*}
		|H_{2,1}(F_f/2)|=\frac{1}{768}\approx 0.0013020833.
	\end{align*}
	\noindent{\bf Case-II.} Let $\tau_1=0$. Then, from \eqref{eq-2.13} we get 
	\begin{align*}  
		|H_{2,1}(F_f/2)|=\bigg|\frac{1}{2304}\left(-36\tau_2^2\right)\bigg|\leq\frac{1}{64}\approx 0.015625.
	\end{align*}
	\noindent{\bf Case-III.} Let $\tau_1\in (0, 1)$. Applying triangle inequality in \eqref{eq-2.13} and using the fact that $|\tau_3|\leq 1$, we obtain
	\begin{align}\label{eq-2.14}
		|H_{2,1}(F_f/2)|&\nonumber\leq\frac{1}{2304}\bigg(\bigg|3\tau_1^4-4\tau_1^2\tau_2(1-\tau_1^2)-12\tau_2^2(1-\tau_1^2)(3+\tau_1^2)\bigg|\\&\nonumber\quad+48\tau_1(1-\tau_1^2)(1-|\tau_2|^2)\bigg)\\&=\frac{1}{48}\tau_1(1-\tau_1^2)\left(\vline\; A+B\tau_2+C\tau_2^2\;\vline+1-|\tau_2|^2\right)\nonumber\\&:=\frac{1}{48}\tau_1(1-\tau_1^2)Y(A, B, C),
	\end{align}
	where
	\begin{align*}
		A=\frac{\tau_1^3}{16(1-\tau_1^2)},\;\; B=-\frac{\tau_1}{12},\;\;\mbox{and}\;\; C=-\frac{(3+\tau_1^2)}{4\tau_1}.
	\end{align*}
	Note that $AC<0$. Hence, we can apply case (ii) of Lemma \ref{lem-2.2} and discuss the following cases.\\
	
	\noindent{\bf Sub-Case III(a).} A simple computation shows that
	\begin{align*}
		-4AC\left(\frac{1}{C^2}-1\right)-B^2&=\frac{\tau_1^2\left(3+\tau_1^2\right)}{16\left(1-\tau_1^2\right)}\left(\frac{16\tau_2^2}{\left(3+\tau_1^2\right)^2}-1\right)-\frac{\tau_1^2}{144}\\&=-\frac{\tau_1^2\left(21-2\tau_1^2\right)}{36\left(3+\tau_1^2\right)}\leq 0
	\end{align*}
	which is true for $\tau_1\in (0, 1)$. However, we see that
	\begin{align*}
		|B|-2(1-|C|)=-2+\frac{3}{2\tau_1}+\frac{7\tau_1}{12}>0
	\end{align*}
	for all $\tau_1\in (0, 1)$. \textit{i.e.,} $|B|>2(1-|C|)$. Hence, 
	\begin{align*}
		Y(A, B, C)\neq 1-|A|+\frac{B^2}{4(1-|C|)}.
	\end{align*}\\
	
	\noindent{\bf Sub-Case III(b).} For $\tau_1\in (0, 1)$, we see that 
	\begin{align*}
		4(1+|C|)^2=4\left(1+\frac{\left(3+\tau_1^2\right)}{4\tau_1}\right)^2>0\; \mbox{and}\; -4AC\left(\frac{1}{C^2}-1\right)=-\frac{\tau_1^2\left(9-\tau_1^2\right)}{16\left(3+\tau_1^2\right)}<0.
	\end{align*}
	Hence,
	\begin{align*}
		0<\frac{\tau_1^2}{144}=B^2<\min\bigg\{4(1+|C|)^2, -4AC\left(\frac{1}{C^2}-1\right) \bigg\}=-\frac{\tau_1^2\left(9-\tau_1^2\right)}{16\left(3+\tau_1^2\right)}<0
	\end{align*}
	which is false. Clearly, 
	\begin{align*}
		Y(A, B, C)\neq 1+|A|+\frac{B^2}{4(1+|C|)}.
	\end{align*}\\
	
	\noindent{\bf Sub-Case III(c).} Next, note that the inequality
	\begin{align*}
		|C|\left(|B|+4|A|\right)-|AB|=\frac{7\tau_1^4+28\tau_1^2+12}{192\left(1-\tau_1^2\right)}\geq 0\; \mbox{for}\; \tau_1\in (0, 1).
	\end{align*}
	This implies that $|C|(|B|+4|A|)\geq |AB|$. Hence, $Y(A, B, C)\neq |A|+|B|-|C|$. \vspace{1.2mm}
	
	For $\tau_1\in (0, 1)$, we see that 
	\begin{align*}
		|AB|-|C|\left(|B|+4|A|\right)=\frac{17\tau_1^4+44\tau_1^2-12}{192\left(1-\tau_1^2\right)}\leq 0
	\end{align*}
	if $ 17\tau_1^4+44\tau_1^2-12\leq 0$. In fact, this inequality holds for $0<\tau_1\leq \tau_1^*:=\sqrt{-\frac{22}{17}+\frac{4\sqrt{43}}{17}}\approx 0.498808$. Thus, by Lemma \ref{lem-2.2}, we have 
	\begin{align*}
		Y(A, B, C)=-|A|+|B|+|C|=\frac{36-20\tau_1^2-19\tau_1^4}{48\tau_1\left(1-\tau_1^2\right)}\; \mbox{for}\; \tau_1\in (0, \tau_1^*].
	\end{align*} 
	Using \eqref{eq-2.14}, we obtain that 
	\begin{align*}
		|H_{2,1}(F_f/2)|\leq \frac{1}{2304}\left(36-20\tau_1^2-19\tau_1^4\right)=\frac{1}{2304}\Phi_1(\tau_1)
	\end{align*}
	where 
	\begin{align*}
		\Phi_1(t):=36-20t^2-19t^4, \;\;0\leq t\leq \tau_1^*.
	\end{align*}
	Since $\Phi_1^{\prime}(t)\leq 0$ for all $t\in [0, \tau_1^*]$, the function $\Phi_1$ is decreasing on $[0, \tau_1^*]$. Hence, we have $\Phi_1(\tau_1)\leq \Phi_1(0)=36$. Therefore, we see that
	\begin{align*}
		|H_{2,1}(F_f/2)|\leq \frac{1}{64}\approx 0.015625.
	\end{align*}
	Our next task is to find the bound of $ 	|H_{2,1}(F_f/2)| $ on $(\tau_1^*, 1)$ using Lemma \ref{lem-2.2}. Henceforth, by a tedious computation, we see that
	\begin{align*}
		Y(A, B, C)=\left(|C|+|A|\right)\sqrt{1-\frac{B^2}{4AC}}=\frac{12-8\tau_1^2-3\tau_1^4}{24\tau_1\left(1-\tau_1^2\right)}\sqrt{\frac{7+2\tau_1^2}{3+\tau_1^2}}.
	\end{align*}
	In view of \eqref{eq-2.14}, we see that 
	\begin{align*}
		|H_{2,1}(F_f/2)|\leq \frac{12-8\tau_1^2-3\tau_1^4}{1152}\sqrt{\frac{7+2\tau_1^2}{3+\tau_1^2}}=\frac{1}{1152}\Phi_2(\tau_1),
	\end{align*}
	where 
	\begin{align*}
		\Phi_2(t):=(12-8t^2-3t^4)\sqrt{\frac{7+2t^2}{3+t^2}},  \;\;\tau_1^*\leq t\leq 1.
	\end{align*}
	We see that 
	\begin{align*}
		\Phi_2^{\prime}(t)=-\frac{t\left(348+452t^2+185t^4+24t^6\right)}{\left(3+t^2\right)^2}\sqrt{\frac{3+t^2}{7+2t^2}}<0\;\;\mbox{for}\;\; \tau_1^*\leq t\leq 1.
	\end{align*}
	By the similar argument being used before, we see that
	\begin{align*}
		\Phi_2(\tau_1)\leq \frac{1}{289}\left(-2192+992\sqrt{43}\right).
	\end{align*}
	Hence,
	\begin{align*}
		|H_{2,1}(F_f/2)|\leq\frac{-137+62\sqrt{43}}{20808}\approx 0.0129547.
	\end{align*}
	Summarizing all the above cases, we conclude that
	\begin{align*}
		|H_{2,1}(F_f/2)|\leq \frac{1}{64}
	\end{align*}
	which is the desired inequality of the result.\vspace{1.2mm}
	
	To establish the sharpness of inequality, we consider the function $f_1\in\mathcal{S}^*_G$ defined by 
	\begin{align*}
		f_1(z)=z\exp\left(\int_{0}^{z}\frac{\Psi(t^2)-1}{t}dt\right)=z+\frac{1}{4}z^3+\cdots.
	\end{align*}
	It is easy to see that $a_2=0$ and $a_3=1/4$. Thus, it follows that 
	\begin{align*}
		|\gamma_1\gamma_3-\gamma_2^2|=\frac{1}{48}\bigg|\left(a_2^4-12a_3^2+12a_2a_4\right)\bigg|=\frac{1}{64}.
	\end{align*}
	This completes the proof.
\end{proof}
\section{\bf Sharp bound of Fekete-Szeg$\ddot{o}$ inequality for the class $\mathcal{S}^*_G$}
This determinant  $H_{q,n}(f)$ was discussed by several authors with $q = 2$. For example, we know that the functional $H_{2,1}(f)=a_3 -a^2_2$ is known as the Fekete-Szeg$\ddot{o}$ functional and they consider the further generalized functional $a_3 -\mu a^2_2$, where $\mu$ is some real number. Estimation for the upper bound of $|a_3 -\mu a^2_2|$ is known as the Fekete-Szeg$\ddot{o}$ problem. In 1969, Keogh and Merkes \cite{Keogh-Merkes-PAMS-1969} solved the Fekete-Szeg$\ddot{o}$ problem for the class $\mathcal{S}^*$. A unified approach to the Fekete-Szeg$\ddot{o}$ problem were investigated by many authors for various subclasses \cite{Abdel-Thomas-PAMS-1992, Elin-Jaco-RM-2022, Fekete-Szego_JLMS-1933, Keogh-Merkes-PAMS-1969,Koepf-PAMS-1987,Xu-Jiang-Liu-JMAA-2023}. In this section we shall investigate the upper bound of $|a_3 -\mu a^2_2|$ for the class $\mathcal{S}^*_G$. The following lemma will be useful in our investigation.
\begin{lem}(\cite{Ma-Minda-1994})\label{lem-3.1}
	Let $p\in\mathcal{P}$ be given by \eqref{eq-2.1}. Then
	\begin{align*}
		|c_2 -vc^2_1|\leq\begin{cases}
			-4v +2 \;\;\;\; v<0,\\ 2 \;\;\;\;\;\;\;\;\;\;\;\;\;\; 0\leq v\leq 1, \\ 4v -2 \;\;\;\;\;\;\; v>1.
		\end{cases}
	\end{align*}
	For $v<0$ or $v>1$, the equality holds if and only if
	\begin{align*}
		h(z)=\frac{1+z}{1-z}
	\end{align*}
	or one of its rotations. If $0 < v < 1$, then the equality is true if and only if
	\begin{align*}
		h(z)=\frac{1+z^2}{1-z^2}
	\end{align*}
	or one of its rotations.
\end{lem}

We obtain the following sharp inequality, which is the upper bound of $|a_3 -\mu a^2_2|$ for the class $\mathcal{S}^*_G$.
\begin{thm}\label{Th-3.1}
	Let $f(z)=z+a_2z^2+a_3z^3+\cdots\in\mathcal{S}^*_G$. Then we have 
	\begin{align*}|a_3-\mu a_2^2|\leq 
		\begin{cases}
			\dfrac{1}{12}(1-3\mu),\;\;\; \mu<-\dfrac{2}{3},\vspace{2mm}\\
			\dfrac{1}{4},\;\;\;\;\;\;\;\;\;\;\;\;\;\;\;\;-\dfrac{2}{3}\leq \mu\leq\dfrac{4}{3},\vspace{2mm}\\
			\dfrac{1}{12}(1-3\mu),\;\;\; \mu>\dfrac{4}{3}.
		\end{cases}.
	\end{align*}
	The inequalities are sharp.
\end{thm}

\begin{proof}
	Since $f\in\mathcal{S}^*_G$, there exists an analytic function $\omega$ with $\omega(0)=0$ and $|\omega(z)|<1$ in $\mathbb{D}$ such that \eqref{eq-2.6} holds. Then in view of \eqref{eq-2.10}, we see that 
	\begin{align*}
		|a_3-\mu a_2^2|=\bigg|\dfrac{1}{24}\left(3c_2-c_1^2\right)-\mu\frac{c_1^2}{16}\bigg|=\frac{1}{8}|c_2-vc_1^2|,
	\end{align*}
	where $v=(2+3\mu)/6$. \\
	
	\noindent{\bf Case 1.} If $v<0$ \emph{i.e.} if $\mu<-(2/3)$, then by Lemma \ref{lem-3.1}, we obtain 
	\begin{align*}
		|c_2-vc_1^2|=-4v+2=\frac{2(1-3\mu)}{3}.
	\end{align*}
	
	\noindent{\bf Case 2.} If $0\leq v\leq 1$ \emph{i.e.} if $-(2/3)\leq \mu\leq (4/3)$, then by Lemma \ref{lem-3.1}, we see that 
	\begin{align*}
		|c_2-vc_1^2|=2.
	\end{align*}
	\noindent{\bf Case 3.} If $v> 1$ \emph{i.e.} if $\mu> (4/3)$, then by Lemma \ref{lem-3.1}, we have
	\begin{align*}
		|c_2-vc_1^2|=4v-2=\frac{2(3\mu-1)}{3}.
	\end{align*}
	Summarizing all the above cases, we obtain the desired inequality of the result. \vspace{2mm}
	
	The next part is to show the sharpness of the inequality. We consider the following functions $f_2$ and $f_3$ from the class $ \mathcal{S}^*_G$
	\begin{align*}
		f_2(z)&=z\exp\left(\int_{0}^{z}\frac{\Psi(t)-1}{t}dt\right)=z+\frac{1}{2}z^2+\frac{1}{12}z^3+\cdots,\vspace{1.5mm}\\
		f_3(z)&=z\exp\left(\int_{0}^{z}\frac{\Psi(t^2)-1}{t}dt\right)=z+\frac{1}{4}z^3+\frac{1}{96}z^5+\cdots.
	\end{align*} 
	For the expression of $f_2$, we see that $a_2=1/2$ and $a_3=1/12$. In either case $\mu<-(2/3)$ or $\mu>(4/3)$, we see that 
	\begin{align*}
		|a_3-\mu a_2^2|=\frac{1}{12}|1-3\mu|=\pm\frac{1}{12}(1-3\mu).
	\end{align*}
	For the expression of $f_3$, we have $a_2=0$ and $a_3=1/4$. When $-(2/3)\leq \mu\leq (4/3)$, we see that 
	\begin{align*}
		|a_3-\mu a_2^2|=|a_3|=\frac{1}{4}.
	\end{align*}
	This establishes sharpness of the inequalities. This completes the proof.
\end{proof}

It is worth noticing that for $\mu=1$, the quantity $|a_3-\mu a_2^2|$ is the Zalcman functional $|a_3-a_2^2|$ (for $n=2$) which is also the Hankel determinant $H_{2,1}(f)$. As a consequence of Theorem \ref{Th-3.1}, we obtain the following corollary in which we show that the sharp bound of Zalcman functional and Hankel determinant $H_{2,1}(f)$ are obtained for the class $\mathcal{S}^*_G$.
\begin{cor}
	Let $f(z)=z+a_2z^2+a_3z^3+\cdots\in\mathcal{S}^*_G$. Then we have 
	\begin{align*}
		|a_3- a_2^2|\leq\frac{1}{4}.
	\end{align*}
	The inequality is sharp for the function $f_3\in\mathcal{S}^*_G$.
\end{cor}

\section{\bf Sharp bound of Zalcman functional for the class $\mathcal{S}^*_G$}
In recent years, special attention has been given on coefficient problems in the class $\mathcal{S}$ and its subclasses, and more broadly, in sub-classes of the class $\mathcal{A}$. In the early $70$s, Lawrence Zalcman posed the conjecture that if $f\in\mathcal{S}$, and is given by \eqref{eq-1.1}, then for $n\geq 2$,
\begin{align*}
	|a^2_{n} -a_{2n-1}|\leq (n-1)^2
\end{align*}
with equality for the Koebe function $K(z):=z/(1-z)^2$ for $z\in \mathbb{D}$, or a rotation. This conjecture implies the celebrated Bieberbach conjecture $|a_n|\leq n$. Bieberbach Theorem shows that the Zalcman conjecture is true for $n=2$ (see \cite[p. 35]{Goodman-1983}). Kruskal established the conjecture for $n=3$ (see \cite{Krushkal-JAM-1995}) and more recently for $n=4,5,6$ (see \cite{Krushkal-GMJ-2010}). For $n>6$ the Zalcman conjecture remains an open problem. In this section, we establish a sharp bound for the Zalcman functional in the case where $n =3$ \textit{i.e.}, we investigate the upper bound of $|a_3^2-a_5|$ for the class $\mathcal{S}^*_G$. The following lemma will be useful in our investigation.\vspace{1.2mm}

\begin{lem}(\cite{Ravichandran-Verma-CRMAS-2015})\label{lem-4.1}
	Let $p\in\mathcal{P}$ be given by \eqref{eq-2.1}. If $0<a<1$, $0 <b<1$ and
	\begin{align*}
		8a(1-a)\{(b\beta -2\lambda)^2 +(b(a+b)-\beta)^2\} +b(1-b)(\beta -2ab)^2 \leq 4ab^2(1-a)(1-b)^2
	\end{align*}
	then 
	\begin{align*}
		|\lambda c^4_1 +ac^2_2 +2bc_1 c_3 -\frac{3}{2}\beta c^2_1 c_2 -c_4|\leq 2.
	\end{align*}
\end{lem}

Now, we obtain the following result \textit{i.e.}, a sharp bound for the Zalcman functional in the case where $n =3$ for the class $\mathcal{S}^*_G$.
\begin{thm}
	Let $f(z)=z+a_2z^2+a_3z^3+\cdots\in\mathcal{S}^*_G$. Then we have 
	\begin{align*}
		|a_3^2-a_5|\leq \frac{1}{8}.
	\end{align*}
	The inequality is sharp.
\end{thm}
\begin{proof}
	Since $f\in\mathcal{S}^*_G$, in view of \eqref{eq-2.10}, we obtain that
	\begin{align}\label{eq-4.1}
		|a_3^2-a_5|=\frac{1}{16}\bigg|\lambda c_1^4+ac_2^2+2bc_1c_3-\frac{3}{2}\beta c_1^2c_2-c_4\bigg|,
	\end{align}
	where $\lambda=91/720$, $a=17/24$, $b=5/12$, and $\beta=109/216$. A tedious computation yields that
	\begin{align*}
		8a(1-a)\{(b\beta-2\lambda)^2+(b(a+b)-\beta)^2\}&+b(1-b)(\beta-2ab)^2-4ab^2(1-a)(1-b)^2\\&=-\frac{253483853}{6046617600}<0.
	\end{align*}
	Also, it is easy to see that $a=17/24<1$ and $b=5/12<1$. Thus, the conditions of Lemma \ref{lem-4.1} are satisfied. In view of Lemma \ref{lem-4.1}, we have
	\begin{align*}
		\bigg|\lambda c_1^4+ac_2^2+2bc_1c_3-\frac{3}{2}\beta c_1^2c_2-c_4\bigg|\leq 2.
	\end{align*}
	Hence, from \eqref{eq-4.1}, it follows that
	\begin{align*}
		|a_3^2-a_5|\leq \frac{1}{8}.
	\end{align*}
	Thus the inequality in the result is obtained.\vspace{2mm}
	
	To show that the inequality is sharp, we consider the function $f_4$ given by 
	\begin{align*}
		f_4(z)=z\exp\left(\int_{0}^{z}\frac{\Psi(t^4)-1}{t}dt\right)=z+\frac{1}{8}z^5+\cdots.
	\end{align*}
	We observe that $a_2=a_3=a_4=0$ and $a_5=\frac{1}{8}$. It is evident that $|a_3^2-a_5|=|-\frac{1}{8}|=\frac{1}{8}$. This completes the proof.
\end{proof}
In $1999$, Ma (see \cite{Ma-JMAA-1999}) generalized the  Zalcman conjecture as: for $f\in\mathcal{S}$, $|a_{n} a_{m} -a_{m+n-1}|\leq (n-1)(m-1)$; $n,m\in \mathbb{N}\setminus\{1\}$. In recent years there has been a great deal of attention devoted to finding sharp bounds of the Zalcman functional $J_{2,3}:=a_2 a_3 -a_4$ for several class of functions (see \cite{Lecko-Sim-RM-2019,Ravichandran- Verma-JMAA-2017,Krushkal-GMJ-2010} and references therein). Now we compute the sharp bounds of the generalized Zalcman functional $J_{2,3}:=a_2 a_3 -a_4$ for the  class $\mathcal{S}^{*}_{G}$ being a special case of the generalized Zalcman functional $J_{n,m}:=a_n a_m -a_{n+m-1}$, $n,m\in\mathbb{N}\setminus\{1\}$, which was investigated by Ma in \cite{Ma-JMAA-1999} for $f\in\mathcal{S}$. The following lemma which will play a key role in proving the sharp inequality. \vspace{1.2mm} 

\begin{lem}(\cite{Ali-BMMSS-2001})\label{lem-4.2}
	Let $p\in\mathcal{P}$ be given by \eqref{eq-2.1} with $0\leq B\leq 1$ and $B(2B-1)\leq D\leq B$. Then 
	\begin{align*}
		|c_3 -2Bc_1 c_2 +Dc^3_1|\leq 2.
	\end{align*}
\end{lem}

We obtain the following result concerning the sharp bound for the generalized Zalcman functional $J_{2,3}$ for the class $\mathcal{S}^{*}_{G}$.
\begin{thm}
	Let $f(z)=z+a_2z^2+a_3z^3+\cdots\in\mathcal{S}^*_G$. Then we have 
	\begin{align*}
		|a_2a_3-a_4|\leq \frac{1}{6}.
	\end{align*}
	The inequality is sharp.
\end{thm}
\begin{proof}
	Since $f\in\mathcal{S}^*_G$, in view of \eqref{eq-2.10}, we have 
	\begin{align*}
		|a_2a_3-a_4|=\frac{1}{12}|c_3-2Bc_1c_2+Dc_1^3|,
	\end{align*}
	where $B=7/12$ and $D=7/24$. It is easy to see that $0\leq B\leq 1$ and the inequality $B(2B-1)\leq D\leq B$ implies that
	\begin{align*}
		\frac{7}{72}\leq\frac{7}{24}\leq\frac{7}{12}.
	\end{align*}
	Hence, by Lemma \ref{lem-4.2}, 
	\begin{align*}
		|a_2a_3-a_4|\leq \frac{1}{6}.
	\end{align*}
	The desired inequality is obtained.\vspace{1.2mm}
	
	To show that the inequality is sharp, we consider the function $f_5$ from the class $\mathcal{S}^*_G$ as follows
	\begin{align*}
		f_5(z)=z\exp\left(\int_{0}^{z}\frac{\Psi(t^3)-1}{t}dt\right)=z+\frac{1}{6}z^4+\cdots.
	\end{align*}
	We observe that $a_2=0$ and $a_3=0$, while $a_4=\frac{1}{6}$. It is evident that $|a_2a_3 - a_4| = \frac{1}{6}$. This completes the proof.
\end{proof}
\vspace{2mm}
\noindent\textbf{Compliance of Ethical Standards:}\\

\noindent\textbf{Conflict of interest.} The authors declare that there is no conflict  of interest regarding the publication of this paper.\vspace{1.5mm}

\noindent\textbf{Data availability statement.}  Data sharing is not applicable to this article as no datasets were generated or analyzed during the current study.

\end{document}